\documentclass[pdflatex,sn-mathphys-num]{sn-jnl}

\usepackage{graphicx}%
\usepackage{multirow}%
\usepackage{amsmath,amssymb,amsfonts}%
\usepackage{amsthm}%
\usepackage{mathrsfs}%
\usepackage[title]{appendix}%
\usepackage{xcolor}%
\usepackage{textcomp}%
\usepackage{manyfoot}%
\usepackage{booktabs}%
\usepackage{algorithm}%
\usepackage{algorithmicx}%
\usepackage{algpseudocode}%
\usepackage{listings}%

\usepackage{mathrsfs} 
\usepackage{enumitem,lineno}
\usepackage{graphicx}

\usepackage[T1]{fontenc}
\usepackage{amsfonts,amsmath,amssymb}
\usepackage{tikz,comment}
\usetikzlibrary{positioning, calc, trees, decorations.markings, patterns, arrows, arrows.meta}
\usepackage{diagmac2}

\theoremstyle{thmstyleone}%
\newtheorem{Theorem}{Theorem}[section]
\newtheorem{Lemma}[Theorem]{Lemma}
\newtheorem{Proposition}[Theorem]{Proposition}
\newtheorem{Corollary}[Theorem]{Corollary}
\newtheorem{Conjecture}[Theorem]{Conjecture}
\newtheorem{Problem}[Theorem]{Problem}

\theoremstyle{thmstyletwo}%
\newtheorem{Example}[Theorem]{Example}
\newtheorem{Remark}[Theorem]{Remark}

\theoremstyle{thmstylethree}%
\newtheorem{Definition}[Theorem]{Definition}

\theoremstyle{thmstyletwo}%

\theoremstyle{thmstylethree}%

\begin{document}

\title[The Gomory-Hu inequality and trees]{The Gomory-Hu inequality and trees}


\author[1,2]{\fnm{Oleksiy} \sur{Dovgoshey}}\email{oleksiy.dovgoshey@gmail.com}

\author*[3]{\fnm{Olga} \sur{Rovenska}}\email{rovenskaya.olga.math@gmail.com}

\affil[1]{\orgdiv{Department of Theory of Function}, \orgname{Institute of Applied Mathematics and Mechanics of NASU},  \city{Sloviansk}, \postcode{84100}, \country{Ukraine}}

\affil[2]{\orgdiv{Department of Mathematics and Statistics}, \orgname{University of Turku}, \city{Turku}, \postcode{20014},  \country{Finland}}

\affil*[3]{\orgdiv{Department of Mathematics and Modeling}, \orgname{Donbas State Engineering Academy},  \city{Kramatorsk}, \postcode{84313},  \country{Ukraine}}


\abstract{Let $G=(V,E)$ be a finite connected graph with vertex set $V$ and edge set $E$, and let $U(G)$ be the set of all ultrametric spaces $(V,d_l)$ generated by vertex labelings $l\colon V \to \mathbb R^+$. We prove that the inequality
\[
|D(V)| \le |E| + 1
\]
holds for all $(V,d_l) \in U(G)$, where $D(V)$ is the distance set of $(V,d_l)$. The necessary and sufficient conditions under which the above inequality turns to an equality are found. Moreover, we prove that each connected graph with non-negative vertex labeling generates a pseudoultrametric space and find some  sufficient conditions under which this space is ultrametric.
}




\maketitle

\section{Introduction}

The theory of ultrametric spaces is closely connected with problems of applied mathematics, physics, linguistics, psychology, and computer science, see, for example,  \cite{RTV86,Mur04a,Lem00,Mur04b,GSP17,SedaHitzler2008,PriessCrampeRibenboim2000}.

By studying the flows in networks, R.~Gomory and T.~Hu~\cite{GH1961S} deduced a statement that can be formulated as follows: If $(X,d)$ is a finite ultrametric space with  {\it distance set}
\[
D(X):=\{d(x,y):x,y\in X\},
\]
then the inequality
\begin{equation}
    \label{dhlu}
|D(X)| \le |X|
\end{equation}
holds.

 We will call inequality \eqref{dhlu} the {\it Gomory--Hu inequality}.

\begin{Definition}\label{kahkas}
    A finite ultrametric space $(X,d)$ is said to be a ${\bf GH}$-space if the equality
\(
|D(X)| = |X|
\)
holds. In what follows we denote by ${\bf GH}$ the class of all ${\bf GH}$-spaces.
\end{Definition}

Two descriptions of the class ${\bf GH}$ were obtained in terms of the representing trees of finite ultrametric spaces and, respectively, so-called diametrical graphs of such spaces (see \cite{PD2014JMS}, Theorem 2.3 and Theorem 3.1). 
A criterion of $(X,d) \in {\bf GH}$ in terms of weighted Hamiltonian cycles and weighted Hamiltonian paths are found in Theorem 2.5 of  paper \cite{DPT2015}. The number of non-isometric spaces $(X,d) \in {\bf GH}$ are also calculated in \cite{DPT2015}.

The finite ultrametric spaces $(X,d)$ with $|X| \ge 3$ for which every self-isometry $X \to X$ has at least $|X|-2$ fixed points  give us an interesting example of ${\bf GH}$-spaces. (See Theorem 2.3 of \cite{PD2014JMS} and Theorem 3.4 of \cite{DPT2017FPTA}.)

At the present paper we
investigate the ultrametric spaces generated by 
  connected  graphs $G$ with given vertex labelings $l : V(G)\to \mathbb{R}^+$. 

The ultrametric spaces generated by labeled trees were introduced in \cite{Dov2020TaAoG}. Some  results describing the metric and topological properties of such spaces can be found in \cite{DCR2025OJAC,DK2024DLPSSAG,DK2022LTGCCaDUS,DR2026MDPI,DR2026PIGC,DR2026AnComb,DR2025USGbLSG,DV2025JMS}.

The present paper is organized as follows. In Section \ref{sec2} we collect together some definitions and facts related to ultrametric spaces and graphs. 

The main results of the paper are presented in Section \ref{sec3}.
Theorem~\ref{tteo5}, the first theorem of Section~\ref{sec3}, shows that every connected labeled graph $G$ generates a pseudoultrametric on $V(G)$. Moreover, we show that this pseudoultrametric is an ultrametric if $G$ is locally finite. The necessary and sufficient conditions under which an ultrametric space generated by finite labeled tree belongs to the class $\bf GH$ are given in Theorem~\ref{tteo6}.
Theorem~\ref{mak} contains a partial generalization of Theorem~\ref{tteo6} to the case of ultrametric spaces generated by arbitrary finite connected labeled graphs. In Proposition~\ref{lero} we prove that every finite connected labeled graph $G(l)$ 
generates a $\bf GH$-space for 
 some suitable vertex labeling $l$.
Example~\ref{hlo} describes the vertex labelings on complete finite   graphs  which generate $\bf GH$-spaces.

The final Section \ref{sec4} contains new conjectures and some open problems.

\section{Preliminaries}\label{sec2}

Recall some necessary definitions from the theory of metric spaces and the graph theory. 

Let $X$ be a non-empty set. An {\it ultrametric} on  $X$ is a function $d : X \times X \to \mathbb{R}^+$, $\mathbb{R}^+ := [0,\infty)$, such that for all $x,y,z \in X$:
\begin{enumerate}
\item[{\it(i)}] $d(x,y)=d(y,x)$,
\item[{\it(ii)}]  $d(x,y)=0 \iff x=y$,
\item[{\it(iii)}]  $d(x,y)\le \max\{d(x,z),d(z,y)\}$.
\end{enumerate}

Inequality $(iii)$ is often called the {\it strong triangle inequality}.

If we replace condition $(ii)$ by the weaker condition
   $ d(x,x)=0,$
then $d$ is called a {\it pseudoultrametric} on $X$.

Just as in the case of an ultrametric space, we introduce  the distance set $D(X)$ of a pseudoultrametric space $(X,d)$ as
\[
D(X) := \{ d(x,y) : x,y \in X \}.
\]

The following proposition directly follows from Theorem 14 in Chapter 4 of book \cite{Kelley1965}.

\begin{Proposition}\label{dxo} Let $(X,d)$ be a pseudoultrametric space. Then there exists $Y \subseteq X$ such that the function
\[
\rho:=d|_{Y\times Y}
\]
is an ultrametric on $Y$ and the equality
\[
D(X) = D(Y)
\]
holds, where $D(X)$ and, respectively, $D(Y)$ are the distance sets of the pseudoultrametric space $(X,d)$ and the ultrametric space $(Y,\rho)$.
\end{Proposition}

Let us turn now to the graph theory.
  Recall that a {\it graph} is a pair $(V,E)$ consisting of a non-empty set $V$ and a (possibly empty) set $E$ whose elements are unordered pairs of distinct points from $V$.
For a graph $G=(V,E)$, the sets $V=V(G)$ and $E=E(G)$ are called the {\it set of vertices} and the {\it set of edges}, respectively. If $x$ and $y$ are vertices of $G$ and $\{x,y\}\in E(G)$, then we say that $x$ and $y$ are {\it adjacent} in $G$. A graph $G$ is said to be {\it empty} if the equality $E(G)=\emptyset$ holds. A graph $H$ is called a {\it subgraph} of a graph $G$, $H \subseteq G$, if the inclusions
\[
V(H) \subseteq V(G) \quad \text{and} \quad E(H) \subseteq E(G)
\] 
hold. A graph $G$ is said to be {\it finite} if this graph contains a finite number of vertices, $|V(G)| < \infty$.

Let $v$ be a vertex of a graph $G$. The cardinal number of the set
\[
\{u \in V(G) : \{u,v\} \in E(G)\}
\]
is the {\it degree} of $v$. The degree of $v\in V(G)$ will be denoted by $\delta_G(v)$,
\begin{equation}
    \label{mmks}
\delta_G(v) := \left|\{u \in V(G) : \{u,v\} \in E(G)\}\right|.
\end{equation}

A graph $G$ is called {\it locally finite} if the degree $\delta_G(v)$ is finite,
\[
\delta_G(v) <\infty,
\]
for every $v\in V(G)$.

The next theorem is a special case of Theorem 1.2 of book \cite{BM2008}.

\begin{Theorem}\label{rghki} Let $G$ be a finite graph. Then the equality
\[
2|E(G)| = \sum_{v \in V(G)} \delta_G(v)
\]
holds.
\end{Theorem}

A {\it path} is a graph $P=(V,E)$ of the form
\begin{equation}
    \label{gasl}
V=\{x_0,x_1,\dots,x_k\}, \quad E=\{\{x_0,x_1\},\dots,\{x_{k-1},x_k\}\},
\end{equation}
where all $x_i$ are distinct and $k \geq 1$. If \eqref{gasl} holds, then we write $P=(x_0,\ldots,x_k)$ and say that $P$ is a path joining the vertices $x_0$ and  $x_k$.  
A graph $G$ is called {\it connected} if for any two distinct vertices $u,v\in V(G)$ there exists a path $P\subseteq G$ joining these vertices. In particular, any graph $G$ with $|V(G)|=1$ is connected.
A finite graph $C$ is a {\it cycle} if $|V(C)|\ge 3$ and there exists an enumeration $v_1,v_2,\dots,v_n$ of its vertices such that
\[
\{v_i,v_j\}\in E(C) \iff (|i-j|=1 \text{ or } |i-j|=n-1).
\]
 A connected graph without cycles is called a {\it tree}.

\begin{Proposition}\label{hfdsxbj} Let $G$ be a finite connected graph. Then the following conditions are equivalent.
\begin{enumerate}
\item[(i)] $G$ is a tree.
\item[(ii)] Any two distinct vertices of $G$ are joined by unique path in $G$.
\item[(iii)] The equality
\[
|V(G)| = 1 + |E(G)|
\]
holds.
\end{enumerate}
\end{Proposition}

See, for example, Theorem 1.5.2 and Corollary 1.5.3 in~\cite{Diestel2017}.

\begin{Theorem}\label{kiii}
     Let $G$ be a connected graph. Then there exists a tree $T \subseteq G$ such that
\[
V(T) = V(G).
\]
\end{Theorem}
For the proof see, for example, Proposition 14 in Chapter 1 of book  \cite{SerreTrees1980}.

Proposition \ref{hfdsxbj} and Theorem \ref{kiii} imply the following corollary.

\begin{Corollary}\label{letp} Let $G$ be a finite connected graph. If the inequality
\begin{equation}
    \label{jako1}
1 + |E(G)| \le |V(G)|
\end{equation}
holds, then $G$ is a tree.
\end{Corollary}

\begin{proof}
Let inequality \eqref{jako1} hold. By Theorem \ref{kiii} there exists a tree $T$ such that
\[
E(T) \subseteq E(G) \quad \text{and} \quad V(T) = V(G).
\]
Consequently, we have the inequality
\begin{equation}
    \label{jako2}
|E(T)| \le |E(G)| 
\end{equation}
and the equality
\begin{equation}
    \label{jako3}
|V(T)| = |V(G)|. 
\end{equation}
Inequality \eqref{jako2}, Proposition \ref{hfdsxbj}, and equality \eqref{jako3} imply
\[
|V(G)| = |V(T)| = 1 + |E(T)| \le 1 + |E(G)|.
\]
Thus the inequality
\[
|V(G)| \le 1 + |E(G)|
\]
holds. The last inequality and \eqref{jako1} imply the equality
\[
|V(G)| = 1 + |E(G)|.
\]
Consequently, $G$ is a tree by Proposition \ref{hfdsxbj}.
\end{proof}

A tree $T$ may have a distinguished vertex $r$ called the {\it root}; in this case $T$ is called a {\it rooted tree} and we write $T=T(r)$.

If $T = T(r)$ is a rooted tree, $v$ is a vertex of $T$, such that $v \neq r$, 
then the {\it level} of $v$ is the number $\operatorname{lev}(v):=|E(P)|$ where $P$ is a path joining $v$ with $r$.

Moreover, the root $r$ of $T$ has, by definition,  the zero level $\operatorname{lev}(r):=0$.

 A non-empty graph $G=(V,E)$ together with a function $w:E\to \mathbb{R}^+$ is called a {\it weighted graph}, and $w$ is called a {\it weight}. The weighted graphs will be denoted by $G(w)$.

\begin{Definition} Let $G(w)$ be a weighted graph.  A weight $w$ is said to be {\it ultrametrizable (pseudoultrametrizable)} if there exists an ultrametric (pseudoultrametric) $\rho_w : V(G) \times V(G) \to \mathbb{R}^+$ such that 
\begin{equation}
\label{kaske}
    w(\{x,y\})=\rho_w(x,y)
\end{equation}
 for each $\{x,y\} \in E(G)$.
\end{Definition}

The next theorem directly follows from Theorems 1--3 of paper \cite{DP2013SM}.

\begin{Theorem}\label{kasya} Let $G(w)$ be a  weighted graph. The weight $w$ is pseudoultrametrizable if and only if for any cycle $C \subseteq G$ there exist at least two different edges $e_1, e_2 \in E(C)$ such that
\begin{equation}
    \label{samf}
w(e_1) = w(e_2) = \max_{e \in E(C)} w(e).
\end{equation}
If $G$ is connected, $w$ is  pseudoultrametrizable, then the function $\rho_w : V(G) \times V(G) \to \mathbb{R}^+$ defined for all $x,y \in V(G)$ as 
\begin{equation}
    \label{kjvdde}
\rho_w(x,y) :=
\begin{cases}
0, & \text{if } x = y,\\[6pt]
\inf\limits_{P \in \mathcal{P}_{x,y}} \left( \max\limits_{e \in E(P)} w(e) \right), & \text{if } x \ne y,
\end{cases}
\end{equation}
where $\mathcal{P}_{x,y}$ is the set of all paths joining  $x$ and $y$ in $G$, is a pseudoultrametric on $V(G)$ and equality \eqref{kaske} holds.
\end{Theorem}

\begin{Remark}\label{kaf}The pseudoultrametric defined by formula \eqref{kjvdde} can be treated as an ultrametric analogue of the {\it shortest path} pseudometric. In particular, Theorem \ref{kasya} is an ultrametric modification of Proposition 2.1 from paper \cite{DMV2013AC}.
\end{Remark}

Let us recall now the concept of a {\it labeled graph.}

\begin{Definition} A labeled graph $G(l)$ is a pair $(G,l)$, where $G$ is a graph and $l$ is a mapping defined on the vertex set $V(G)$.
\end{Definition}

We will consider
the labeled graphs
$G(l)$ with non-negative labelings $l\colon V(G) \to \mathbb R^+$ only.

Let $G = G(l)$ be a connected labeled graph. Similarly to \cite{Dov2020TaAoG}, we define a mapping $d_l : V(G) \times V(G) \to \mathbb{R}^+$ as
\begin{equation}
    \label{fdes1}
d_l(x,y) := 
\begin{cases}
0, & \text{if } x = y,\\
\inf\limits_{P\in \mathcal{P}_{x,y}}\left(\max\limits_{v\in V(P)} l(v)\right), & \text{otherwise},
\end{cases}
\end{equation}
where $\mathcal{P}_{x,y}$ is the set of all paths  joining $x$ and $y$ in $G$.

In what follows, we will say that a space $(X,d_l)$ is {\it generated by 
labeled graph} $G (l)$.

If a graph $G$ is a tree, then, by Proposition \ref{hfdsxbj}, for each two distinct $x,y \in V(G)$ there is a unique path $P\in \mathcal{P}_{x,y}$.
So we can rewrite formula \eqref{fdes1} in the form
\begin{equation}
    \label{fdes5}
d_l(x,y) := 
\begin{cases}
0, & \text{if } x = y, \\
\max\limits_{v \in V(P)} l(v), & \text{otherwise}.
\end{cases}
\end{equation}

The following result is a direct corollary of Proposition 3.2 from \cite{Dov2020TaAoG}.

\begin{Theorem}\label{t2t} Let $T = T(l)$ be a labeled tree and let $d_l \colon V(T)\times V(T) \to \mathbb R^+$ be defined by \eqref{fdes5}. Then the function $d_l$ is an ultrametric on $V(T)$ if and only if the inequality
\begin{equation*}
\max\{l(u), l(v)\} > 0
\end{equation*}
holds for every  $\{x,y\}\in V(T)$.
\end{Theorem}

\section{HG-spaces generated by labeled graphs}\label{sec3}

The following lemma allows us to move from the vertex labelings $l \colon V(G) \to \mathbb R^+$ to the weights $w \colon E(G) \to \mathbb R^+$.

\begin{Lemma}\label{fruj} 
Let $G = G(l)$ be a connected labeled graph,
let $d_l : V(G) \times V(G) \to \mathbb{R}^+$ be generated by $G(l)$, and
let a weight $w : E(G) \to \mathbb{R}^+$ be defined as
\begin{equation}
    \label{00*}
w(\{u,v\}) := \max\{l(u), l(v)\}
\end{equation}
for each $\{u,v\} \in E(G)$. Then the equality
\begin{equation}
    \label{zero}
d_l(x,y)
=
\inf_{P \in \mathcal{P}_{x,y}} \left( \max_{e \in E(P)} w(e) \right)
\end{equation}
holds for all different $x,y\in V(G)$, where $\mathcal{P}_{x,y}$ is the set of all paths joining $x$ and $y$ in $G$.
\end{Lemma}

\begin{proof}
Let $x$ and $y$ be distinct vertices of $G$ and let $P \in \mathcal{P}_{x,y}$. Then the vertices of $P$ can be numbered such that
\begin{equation}
    \label{**1}
P=(v_0,\dots,v_k)
\end{equation}
with $v_0=x$ and $v_k=y$.
Using \eqref{00*} and \eqref{**1} we obtain the equalities
\[
\max_{v \in V(P)} l(v)
=
\max_{0 \le i \le k} l(v_i)
=
\max_{0 \le i \le k-1} \max\{l(v_i), l(v_{i+1})\}
\]
\[
=
\max_{0 \le i \le k-1} w(\{v_i,v_{i+1}\})
=
\max_{e \in E(P)} w(e).
\]
Thus the equality
\begin{equation}
    \label{**2}
\max_{v \in V(P)} l(v) = \max_{e \in E(P)} w(e)
\end{equation}
holds for every $P \in \mathcal{P}_{x,y}$. Now equality \eqref{zero} follows from \eqref{**2}. 
\end{proof}

Equality \eqref{zero} can be rewritten in a much more simple form if $\{x,y\} \in E(G)$.

\begin{Corollary}\label{btoy} Let $G = G(l)$ be a connected labeled graph with  vertex labeling $l : V(G) \to \mathbb R^+,$ and let $d_l : V(G)\times V(G)\to \mathbb R^+$ be generated by $G(l)$. Then the equality
\begin{equation}
    \label{a1}
d_l(x,y)
= \max \{ l(x), l(y) \}
\end{equation}
holds for each edge $ \{x,y\} \in E(G)$.
\end{Corollary}

\begin{proof}
Let $x,y$  be distinct vertices of $G$.
It is clear that $x$  and $ y$  are vertices of every path $ P \in \mathcal{P}_{x,y}$.
Hence the inequality
\begin{equation}
    \label{a2}
\inf_{P \in \mathcal{P}_{x,y}} \left( \max_{v \in V(P)} l(v) \right)
\ge \max \{ l(x), l(y) \}
\end{equation}
holds.
Suppose that $x$ and $y$ are adjacent in $G$. Then the two-point path $(x,y)$ belongs to $\mathcal{P}_{x,y}$ and, consequently, we have the inequality
\begin{equation}
    \label{a3}
\inf_{P \in \mathcal{P}_{x,y}} \left( \max_{v \in V(P)} l(v) \right)
\le \max \{ l(x), l(y) \}.
\end{equation}

Inequalities  \eqref{a2} and \eqref{a3} imply the equality
\begin{equation}
    \label{ae}
\inf_{P \in \mathcal{P}_{x,y}} \left( \max_{v \in V(P)} l(v) \right)
= \max \{ l(x), l(y) \}.
\end{equation}

Equality \eqref{a1} follows from \eqref{fdes1} and \eqref{ae}.
\end{proof}

The following theorem is a partial generalization of Theorem \ref{t2t}.

\begin{Theorem}\label{tteo5} Let $G = G(l)$ be a connected  labeled graph  and let  $d_l : V(G)\times V(G) \to \mathbb{R}^{+}$
be generated by $G(l)$.
Then $d_l$ is  a pseudoultrametric on $V(G)$. Moreover, if $G$ is locally finite, then $d_l$ is an ultrametric if and only if the inequality
\begin{equation}\label{eq:10}
\max\{l(u),l(v)\} > 0
\end{equation}
holds for every edge $\{u,v\} \in E(G)$.
\end{Theorem}

 \begin{proof}

Let us define a weight
\(
w : E(G) \to \mathbb{R}^+
\)
as
\begin{equation}
    \label{trean1}
w(\{u,v\}) := \max \{ l(u), l(v) \}
\end{equation}
for each $\{u,v\} \in E(G).$

First we claim that the weight $w$ is pseudoultrametrizable.

To prove this claim we will use Theorem \ref{kasya}.

Let us consider an arbitrary cycle $C \subseteq G.$
By Theorem \ref{kasya}, $w$ is pseudoultrametrizable if and only if we can find
different $e_1,$ $e_2 \in E(C)$ such that
\begin{equation}
    \label{uro1}
w(e_1) = w(e_2) = \max_{e \in E(C)} w(e).
\end{equation}

Since $C$ is a finite graph, there exists $v_0 \in V(C)$ such that
\begin{equation}
    \label{s1}
l(v_0) = \max_{u \in V(C)} l(u). 
\end{equation}
It follows from the definition of the cycles that there are two distinct vertices $u_1, u_2 \in V(C)$ such that $\{u_1, v_0\}$ and $\{u_2, v_0\}$ belong to $E(C)$. Equalities \eqref{trean1} and \eqref{s1} imply that
\[
w(\{u_1, v_0\}) = l(v_0) = w(\{u_2, v_0\}).
\]
Consequently \eqref{uro1} holds with $e_1 = \{u_1, v_0\}$, $e_2 = \{u_2, v_0\}$.
Thus $w : E(G) \to \mathbb{R}^+$ is pseudoultrametrizable by Theorem \ref{kasya}.
Theorem \ref{kasya} implies also that the mapping $\rho_w : V(G) \times V(G) \to \mathbb{R}^+$ 
defined by formula \eqref{kjvdde} is a pseudoultrametric on $V(G)$. 
Equality \eqref{zero} implies the equality $d_l=\rho_w$. Thus $d_l$ also is a pseudoultrametric on $V(G)$.

 Let $G$ be locally finite.
 Suppose that $d_l$ is an ultrametric on $V(G)$.
Let $\{u,v\}$ be an edge of $G$, 
\begin{equation*}
\{u,v\} \in E(G).
\end{equation*}
Then inequality \eqref{eq:10}  follows from Corollary \ref{btoy} and the definition of ultrametrics.

Suppose now that inequality \eqref{eq:10} holds for all $\{u,v\} \in E(G).$
We must show that $d_l$ is an ultrametric on $V(G).$

Suppose contrary that $d_l$ is not an ultrametric.
Then, using the definitions of the ultrametric spaces and pseudoultrametric ones, we can find different $u^*, v^* \in V(G)$ such that
\begin{equation}
\label{s4}
d_l(u^*,v^*)=0.
\end{equation}

If the inequality
\[
\max\{l(u^*),l(v^*)\}>0
\]
holds, then using \eqref{fdes1} with \( x=u^*\) and \( y=v^* \), we obtain
\[
d_l(u^*,v^*)>0,
\]
contrary to \eqref{s4}. Thus the equalities 
\begin{equation}
    \label{s5}
l(u^*)=0=l(v^*)
\end{equation}
hold.

Let \( N(u^*) \) be the set of all vertices of \( G \) which are adjacent with \( u^* \),
\[
N(u^*)=\{u\in V(G): \{u,u^*\}\in E(G)\}.
\]
Recall that the degree \( \delta_G(u^*) \) is the cardinal number of the set \( N(u^*) \),
\[
\delta_G(u^*)=|N(u^*)|.
\]
The graph \( G \) is locally finite. Consequently, the degree $\delta_G(u^*)$ is finite. Hence the set \( N(u^*) \) is finite,
\begin{equation}
    \label{s6}
|N(u^*)|<\infty.
\end{equation}
Moreover, since \eqref{eq:10} holds for all \( \{u,v\}\in E(G) \), double equality \eqref{s5} implies the inequality 
\[ 
l(u)>0 
\] 
for each \( u\in N(u^*) \).
The last inequality and \eqref{s6} give us the inequality
\begin{equation}
    \label{s7}
\inf_{u\in N(u^*)} l(u) > 0.
\end{equation}
Now using \eqref{s7} and formula \eqref{fdes1} we can prove that
\[
d_l(u^*,v^*) \ge \inf_{u\in N(u^*)} l(u) > 0,
\]
contrary to \eqref{s4}. 

Thus, if $G$ is locally finite and \eqref{eq:10} holds for each $\{u,v\}\in E(G),$ then pseudoultrametric \( d_l: V(G)\times V(G)\to \mathbb{R}^+ \) is an ultrametric on \( V(G) \) as required.

The proof is completed.
\end{proof}

\begin{Definition}\label{maks}
    Let $G=G(l)$ be a connected labeled graph. The 
labeling 
\(
l : V(G) \to \mathbb{R}^+
\)
is {\it non-degenerate} if the inequality
\[
\max \{ l(u), l(v) \} > 0
\]
holds for each $\{u,v\} \in E(G)$.
\end{Definition}

Theorem \ref{tteo5} implies the following corollary.

\begin{Corollary}\label{esghkl}
    Let $G = G(l)$ be a connected finite labeled graph  and let a mapping $d_l $
be generated by $G(l)$.
Then $d_l$ is an ultrametric on $V(G)$ if and only if the labeling 
\(
l : V(G) \to \mathbb{R}^+
\)
is {\it non-degenerate}.
\end{Corollary}

\begin{figure}[ht]
\centering

\begin{minipage}[c]{0.48\textwidth}
\centering
\begin{tikzpicture}[x=0.38pt,y=-0.38pt, line cap=round, line join=round]

\node[font=\small\itshape] at (350,105) {$G$};

\coordinate (u1) at (282,385);
\coordinate (u2) at (648,382);
\coordinate (v1) at (466, 45);
\coordinate (v2) at (466,120);
\coordinate (v4) at (466,231);
\coordinate (vn) at (467,307);

\draw[line width=0.9pt] (u1) -- (v1) -- (u2);
\draw[line width=0.9pt] (u1) -- (v2) -- (u2);
\draw[line width=0.9pt] (u1) -- (v4) -- (u2);
\draw[line width=0.9pt] (u1) -- (vn) -- (u2);

\fill (u1) circle (1.6pt);
\fill (u2) circle (1.6pt);
\fill (v1) circle (1.6pt);
\fill (v2) circle (1.6pt);
\fill (v4) circle (1.6pt);
\fill (vn) circle (1.6pt);

\fill (466,170-8) circle (0.4pt);
\fill (466,170)   circle (0.4pt);
\fill (466,170+8) circle (0.4pt);

\node[font=\small] at (274,399) {$x$};
\node[font=\small] at (656,397) {$y$};

\node[font=\small] at (466, 30) {$z_{n+1}$};

\fill (466,-7-8) circle (0.4pt);
\fill (466,-7)   circle (0.4pt);
\fill (466,-7+8) circle (0.4pt);

\node[font=\small] at (466,105) {$z_n$};
\node[font=\small] at (466,216) {$z_2$};
\node[font=\small] at (468,292) {$z_1$};

\end{tikzpicture}
\end{minipage}
\hfill
\begin{minipage}[c]{0.48\textwidth}
\centering
\begin{tikzpicture}[every node/.style={font=\small}]

\coordinate (xmna) at (-2.5,0);
\coordinate (xmn)  at (-1.3,0);
\coordinate (x0)   at (0,0);
\coordinate (xn)   at (1.3,0);
\coordinate (xnp)  at (2.5,0);

\draw[dashed, line width=0.9pt] (-3.2,0) -- (xmna);
\draw[line width=0.9pt] (xmna) -- (xnp);
\draw[dashed, line width=0.9pt] (xnp) -- (3.2,0);

\fill (xmna) circle (1.6pt);
\fill (xmn)  circle (1.6pt);
\fill (x0)   circle (1.6pt);
\fill (xn)   circle (1.6pt);
\fill (xnp)  circle (1.6pt);

\node[below=3pt] at (xmna) {$x_{-n-1}$};
\node[below=3pt] at (xmn)  {$x_{-n}$};
\node[below=3pt] at (x0)   {$x_{0}$};
\node[below=3pt] at (xn)   {$x_{n}$};
\node[below=3pt] at (xnp)  {$x_{n+1}$};

\node at (-0.7,0.35) {$\cdots$};
\node at (0.7,0.35) {$\cdots$};

\coordinate (y) at (0,4.5);
\fill (y) circle (1.6pt);
\node[above=3pt] at (y) {$y$};

\draw[dashed, line width=0.9pt] (y) -- (-3.2,0);
\draw[dashed, line width=0.9pt] (y) -- (3.2,0);

\draw[line width=0.9pt] (y) -- (xmna);
\draw[line width=0.9pt] (y) -- (xmn);
\draw[line width=0.9pt] (y) -- (x0);
\draw[line width=0.9pt] (y) -- (xn);
\draw[line width=0.9pt] (y) -- (xnp);

\node at (-2.0,3.5) {$H$};

\end{tikzpicture}
\end{minipage}

\caption{$G$ is the union of two star-graphs. $H$ is the union of a star-graph and a double ray.}
\label{f3}
\end{figure}
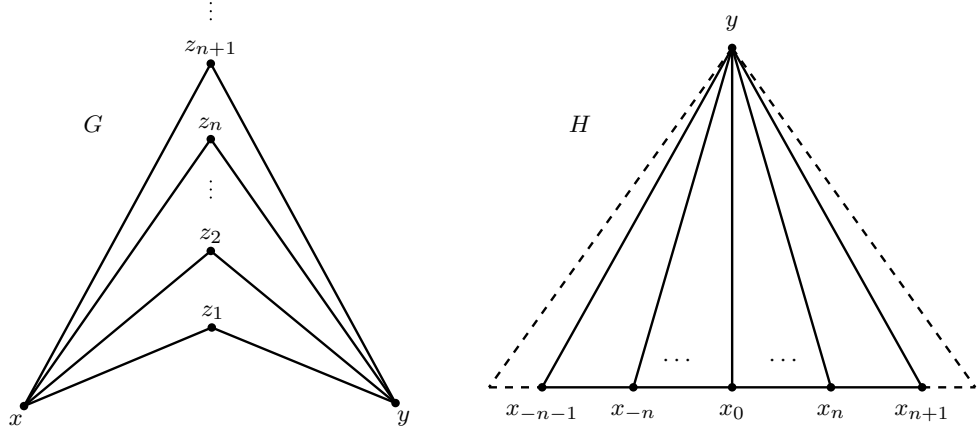

Let us consider now an example of a labeled graph $G = G(l)$ for which the pseudoultrametric $d_l$ is not an ultrametric but
the labeling $l :V(G)\to \mathbb{R}^+$ is non-degenerate.

\begin{Example}\label{rdwwdd} Let $G$ be the graph depicted by Figure \ref{f3},
\[
V(G):=\{x, y, z_1, z_2, \dots\,z_n, z_{n+1},\ldots\},
\quad
E(G) := \bigcup_{n=1}^\infty \{x, z_n\} \cup \bigcup_{n=1}^\infty \{y, z_n\}.
\]
Let us define a labeling $l : V(G)\to \mathbb{R}^+$ as
\[
l(x) := l(y) := 0
\quad
\text {and} \quad 
l(z_n) := \frac{1}{n} 
\]
for each $n \in \mathbb N$.
Then $l$ is a non-degenerate labeling, but
 we have
\[
d_l(x,y) = \inf_{P \in \mathcal{P}_{x,y}} \left(\max_{v \in V(P)} l(v)\right)
= \inf_{n \in \mathbb{N}} \max\{l(x), l(z_n), l(y)\}
\]
\[
= \inf_{n \in \mathbb{N}} \max\left\{0, \frac{1}{n}\right\}
= \inf_{n \in \mathbb{N}} \frac{1}{n} = 0.
\]
Thus $d_l$ is not an ultrametric.
\end{Example}

Let us consider an example of a connected graph $H$ that is not locally finite but generates an ultrametric $d_l$ for each non-degenerate $l : V(H) \to \mathbb{R}^+$.

\begin{Example}\label{hau} Let $H$ be the graph depicted in Figure~\ref{f3},
\[
V(H):=\{y\}\cup\left(\bigcup_{n \in \mathbb Z} \{x_n\}\right),
\quad
E(H) := \bigcup_{n \in \mathbb{Z}} \{x_n, x_{n+1}\} \cup \bigcup_{n \in \mathbb{Z}} \{y, z_n\}.
\]
Then $\delta_H(y) = |\mathbb{Z}|$ holds but a simple modification of the proof of Theorem \ref{tteo5} shows that $d_l$ is an ultrametric for each non-degenerate labeling $l : V(H) \to \mathbb{R}^+$.
\end{Example}

The following theorem gives some necessary and sufficient conditions under which an ultrametric space generated by finite labeled tree belongs to the class ${\bf GH}$.

\begin{Theorem}\label{tteo6} Let $T =T (l)$ be a finite labeled tree with $|V(T)|\geq 2$, let the labeling $l : V(T) \to \mathbb{R}^+$ be non-degenerate, let 
\[
w(\{u,v\}) := \max\{l(u), l(v)\}
\]
for each $\{u,v\} \in E(T)$, and let $d_l$ be generated by $T(l)$.
Then  the following statements are equivalent:
\begin{enumerate}
\item[(i)] $(V(T), d_l) \in {\bf GH}$.
\item[(ii)] The mapping $w : E(T) \to \mathbb{R}^+$ is injective.
\item[(iii)] The distance set $D(V(T))$ of $(V(T),d_l)$ satisfies the equality  
\begin{equation}
    \label{k1}
|D(V(T)|=|E(T)|+1.
\end{equation}
\item[(iv)] The distance set $D(V(T))$ of $(V(T),d_l)$ satisfies the equality
\begin{equation*}
2|D(V(T)|=2+\sum\limits_{v\in V(T)} \delta_T(v).
\end{equation*}
\end{enumerate}
\end{Theorem}

\begin{proof}
$(i) \Rightarrow (ii)$. Let $(V(T), d_l) \in {\bf GH}$ hold.

Let us consider the distance set $D(V(T))$ of the ultrametric space $(V(T), d_l)$,
\[
D(V(T)) = \{ d_l(x,y) : x,y \in V(T) \}.
\]
Then the equality
\[
|V(T)| = |D(V(T))|
\]
holds by Definition~\ref{kahkas}.

Since $l: \in V(T) \to \mathbb R^+$ is non-degenerate, Corollary \ref{btoy} implies the inclusion 
\begin{equation}\label{j1}
w(E(T)) \subseteq D(V(T))\setminus\{0\}.
\end{equation}
Moreover, using \eqref{fdes1}, we obtain
the converse inclusion
\[
w(E(T)) \supseteq D(V(T)) \setminus \{0\}.
\]
Thus the equality
\begin{equation}
    \label{g1}
w(E(T)) = D(V(T)) \setminus \{0\}
\end{equation}
holds.

Since the tree $T$ is finite, the mapping 
$w : E(T) \to \mathbb{R}^+$
is injective if and only if the equality
\begin{equation}
    \label{g2}
|w(E(T))| = |E(T)|
\end{equation}
holds.

Equality \eqref{g1} and the membership relation 
$(V(T), d_l) \in {\bf GH}$ imply the equality
\begin{equation}
    \label{g3}
|w(E(T))| = |V(T)| - 1
\end{equation}
by Definition \ref{kahkas}.

Since $T$ is a tree we also have
\begin{equation}
    \label{g4}
    |E(T)| = |V(T)| - 1 
\end{equation}
by Proposition \ref{hfdsxbj}. Equality \eqref{g2} follows from \eqref{g3} and \eqref{g4}. Statement $(ii)$ is proved.

$(ii) \Rightarrow (i)$. 
Let $w : E(T) \to \mathbb{R}^+$ be injective. Corollary \ref{btoy} implies  inclusion \eqref{j1}.
Consequently the inequality
\[
|D(V(T))| \ge |E(T)| + 1
\]
holds.
Since $T$ is a tree, the last inequality and Proposition \ref{hfdsxbj} give us the inequality
\begin{equation}
    \label{kam1}
|D(V(T))| \ge |V(T)|.
\end{equation}

The Gomory--Hu inequality \eqref{dhlu} with $(X,d) = (V(T), d_l)$ implies the  inequality
\begin{equation}
    \label{kam2}
|D(V(T))| \le |V(T)|.
\end{equation}
Hence the equality
\begin{equation}
    \label{hruh}
|D(V(T))| = |V(T)|
\end{equation}
holds. The last equality implies statement $(i)$ by Definition \ref{kahkas}.

$(i) \Leftrightarrow (iii)$. By Definition \ref{kahkas}, statement $(i)$ holds if and only if  we have equality \eqref{hruh}. Proposition \ref{hfdsxbj} implies that equalities \eqref{k1} and \eqref{hruh} are equivalent. Thus the logical equivalence $(i) \Leftrightarrow (iii)$ is valid. 

$(iii) \Leftrightarrow (iv)$. The validity of the equivalence 
$(iii) \Leftrightarrow (iv)$ follows from Theorem \ref{rghki}.

The proof is completed. 
\end{proof}

\begin{Lemma}\label{ifff}
Let $T$ be a finite tree. Then there exists a non-degenerate labeling 
\(
l : V(T) \to \mathbb{R}^+
\)
such that 
\begin{equation}
    \label{uutt}
    (V(T), d_l) \in {\bf GH},
\end{equation}
and  the equality
\begin{equation}
    \label{fni}
d_l(x,y)= \max\{l(x), l(y)\}
\end{equation}
holds for any two distinct $x$, $y\in V(T)$.
\end{Lemma}

\begin{proof}
Let $r$ be a fixed vertex of the tree $T$. Then there are a non-negative integer number $n$ and disjoint non-empty subsets $V_0,\ldots,V_n$ of the vertex set $V(T)$ such that
\begin{equation}
    \label{r1}
V(T) = \bigcup_{i=0}^{n} V_i
\end{equation}
and the equality
\begin{equation}
    \label{r2}
\operatorname{lev}(v) = i
\end{equation}
holds for every $v \in V_i$ and each $i \in \{0, \ldots, n\}$, where $\operatorname{lev}(v)$ is the level of the vertex $v$ in the rooted tree $T(r)$.

Let
\(
l : V(T) \to \mathbb{R}^+
\)
be an injective labeling and let the inequality
\[
l(v) < l(u)
\]
hold whenever $\operatorname{lev}(v) < \operatorname{lev} (u)$.
Then it is easy to see that the mapping $w : E(T) \to \mathbb{R}^+$,
\[
w(\{u,v\}): = \max\{l(u), l(v)\}
\]
for each $\{v,u\} \in E(T)$, also is injective.

Consequently \eqref{uutt} holds by Theorem \ref{tteo6}.

Let us prove equality \eqref{fni}.
Let $x$ and $y$ be two distinct vertices of $T$ and let
\(
P_x = (x,\dots,r), \) and \(P_y = (y,\dots,r)
\)
be the paths joining $x$ with  $r$ and, respectively, $y$ with $r$. Let us consider the graph $P_x \cup P_y$,
\begin{equation}
    \label{n0}
V(P_x \cup P_y) := V(P_x) \cup V(P_y), 
\end{equation}
\begin{equation}
        \label{n00}
E(P_x \cup P_y) := E(P_x) \cup E(P_y). 
\end{equation}
Then the inequality
\begin{equation}
    \label{n1}
\operatorname{lev}(v) < \max\{\operatorname{lev}(x), \operatorname{lev}(y)\}
\end{equation}
holds for each $v \in V(P_x \cup P_y)\setminus\{x,y\}$.

Let $P_T$ be the path joining $x$ and $y$ in $T$. Since the graph $P_x \cup P_y$ is connected, there is a path $P^*$ joining $x$ and $y$ in $P_x \cup P_y$,
\begin{equation}
    \label{n2}
P^* \subseteq P_x \cup P_y.
\end{equation}
It follows from \eqref{n0}  and \eqref{n00} that $P_x \cup P_y$ is a subgraph of $T$,
\begin{equation}
\label{n3}
P_x \cup P_y \subseteq T. 
\end{equation}
Now \eqref{n2} and \eqref{n3} imply the inclusion
\begin{equation}
\label{n4}
P^* \subseteq T. 
\end{equation}
Since $T$ is a tree, statement $(ii)$ of Proposition \ref{hfdsxbj} and inclusion \eqref{n4} give us the equality
\[
P^* = P_T.
\]
The last equality and inequality \eqref{n1} imply the inequality 
\begin{equation}
    \label{hasd}
\max_{v \in V(P_T)} l(v)\leq \max \{l(x),l(y)\}.
\end{equation}
Inequality \eqref{hasd} and equality \eqref{samf} give us the inequality
\[
d_l(x,y) \le \max \{ l(x), l(y) \}.
\]
The converse inequality
\[
d_l(x,y) \ge \max \{ l(x), l(y) \}
\]
follows directly from \eqref{samf}.
Equality \eqref{fni} is proven.

The proof is completed. 

\end{proof}

The following theorem gives us a partial generalization of Theorem \ref{tteo6}.

\begin{Theorem}\label{mak} Let $G$ be a finite connected graph. Then, for each labeling
$l: V(G) \to \mathbb{R}^+$, we have the inequality
\begin{equation}
    \label{h1}
|D(V(G))| \le |E(G)| + 1,
\end{equation}
where $D(V(G))$ is the distance set of the pseudoultrametric space $(V(G), d_l)$ with $d_l : V(G)\times V(G) \to \mathbb R^+$ generated by $G(l)$. Moreover,
the following statements are equivalent:
\begin{enumerate}
    \item[(i)] There is a labeling $l : V(G) \to \mathbb{R}^+$ such that
\begin{equation}
    \label{h2}
|D(V(G))| = |E(G)| + 1.
\end{equation}
\item[(ii)] The graph $G$ is a tree.
\end{enumerate}
\end{Theorem}

\begin{proof}
Let us prove inequality \eqref{h1}.

By Proposition \ref{dxo} we can find $Y \subseteq V(G)$ such that
\(\rho=d_l|_{Y\times Y}\) is an ultrametric on $Y$ and the equality
\begin{equation}
    \label{sddefh}
D(V(G)) = D(Y)
\end{equation}
holds, 
where $D(Y)$ is the distance set of $(Y,\rho)$. The set $Y$ is a subset of $V(G)$, that implies the inequality
\[
|Y| \le |V(G)|.
\]
The last inequality and condition $(ii)$ of 
 Proposition \ref{hfdsxbj} give us  the inequality
\begin{equation}
    \label{zz1}
|Y| \le 1 + |E(G)|.
\end{equation}
Since the space $(Y,\rho)$ is ultrametric, we have
\begin{equation}
    \label{zz2}
|D(Y)| \le |Y|
\end{equation}
by  the Gomory--Hu inequality. Now using \eqref{sddefh}--\eqref{zz2}, we obtain
\[
|D(V(G))| = |D(Y)| \le |Y| \le 1 + |E(G)|.
\]
Inequality \eqref{h1} follows.

Let us prove that statements $(i)$ and $(ii)$ are equivalent.

 $(i) \Rightarrow (ii)$. Suppose that there is a labeling
\(
l : V(G) \to \mathbb{R}^+
\)
such that equality \eqref{h2} holds for the distance set $D(V(G))$ of the space $(V(G), d_l)$. 
 
We must show that $G$ is a tree. 
By Corollary \ref{letp} the graph $G$ is a tree if the inequality
\begin{equation}
    \label{lewr}
|V(G)| \geq 1 + |E(G)|
\end{equation}
holds.
By Theorem \ref{kiii} we can find a tree $T$ such that
\begin{equation}
    \label{mhl}
V(T) = V(G) \quad \text{and} \quad E(T) \subseteq E(G).
\end{equation}

Equalities \eqref{h2} and \eqref{sddefh} imply the equality
\begin{equation}
\label{fti}
|D(Y)| = |E(G)| + 1.
\end{equation}
The space $(Y, \rho)$ is an ultrametric space. Consequently, we have \eqref{zz2} by Gomory--Hu inequality. Inequality \eqref{zz2}, equality \eqref{fti} and the inclusion $Y \subseteq V(G)$ give us
\[
|E(G)| + 1 = |D(Y)| \le |Y| \le |V(G)|.
\]
Consequently inequality \eqref{lewr} holds. Thus $G$ is a tree as required.

 $(ii) \Rightarrow(i)$. Let $G$ be a tree. Then, by Lemma \ref{ifff}, there is a labeling
\(
l : V(G) \to \mathbb{R}^+
\)
such that
\begin{equation}
    \label{o1}
(V(G), d_l) \in {\bf GH},
\end{equation}
where $d_l$ is an ultrametric on $V(G)$ generated by $G(l)$.

Formula \eqref{o1} and statements $(i)$, $(iii)$ of Theorem \ref{tteo6} with $T(l)=G(l)$ imply equality \eqref{h2} as required.

The proof is completed. 
\end{proof}

Theorems \ref{rghki} and \ref{mak} imply the following corollary.

\begin{Corollary}\label{ter}
Let $G$ be a finite connected graph.
Then the following statements are equivalent:
\begin{enumerate}
    \item[(i)] There is a labeling $l : V(G) \to \mathbb{R}^+$ such that
\[
2|D(V(G)|= 2+\sum\limits_{v \in V(G)} \delta_G(v),
\]
where  $D(V(G))$ is the distance set of the pseudoultrametric space $(V(G), d_l)$ with $d_l : V(G)\times V(G) \to \mathbb R^+$ generated by $G(l)$.
\item[(ii)] The graph $G$ is a tree.
\end{enumerate}
\end{Corollary}

The next proposition follows from Theorem \ref{kiii} and Lemma \ref{ifff}. 

\begin{Proposition}\label{lero}
Let $G$ be a finite connected graph. Then there exists a non-degenerate labeling 
\(
l : V(G) \to \mathbb{R}^+
\)
such that 
\begin{equation}
    \label{uutt-1}
    (V(G), d_{l}) \in {\bf GH},
\end{equation}
where $d_{l}$ is an ultrametric on $V(G)$ generated by $G(l)$.
\end{Proposition}

\begin{proof}
By Theorem \ref{kiii} there is a tree $T$ such that
\(
E(T) \subseteq E(G)\) and \(V(T)=V(G).
\)
Using Lemma \ref{ifff} we can find 
\(
l: V(T) \to \mathbb{R}^+
\)
such that 
\begin{equation}
    \label{d1}
    (V(T), d_l) \in {\bf GH},
\end{equation}
where
\(
d_l : V(T)\times V(T)\to \mathbb{R}^+
\)
is an ultrametric on $V(T)$ generated by labeled tree $T(l)$.

Let us define a labeling \(
l : V(G)\to \mathbb{R}^+
\) as
$l^*(v) := l(v)$
for every $v \in V(G)$. We claim that the ultrametric $d_{l^*} : V(G) \times V(G) \to \mathbb{R}^+$ generated by $G(l^*)$ coincides with the ultrametric $d_l : V(T) \times V(T) \to \mathbb{R}^+$ generated by $T(l)$.

Let us consider an arbitrary pair of distinct $x,y \in V(G)$.  
We must show that
\begin{equation}
    \label{mam59}
d_l(x,y) = d_{l^*}(x,y). 
\end{equation}
Since $l^*(v) = l(v)$ holds for each $v \in V(G)$, the inclusion $E(T) \subseteq E(G)$ and formula \eqref{fdes1} imply the inequality
\begin{equation*}
d_l(x,y) \ge d_{l^*}(x,y). 
\end{equation*}
Thus \eqref{mam59} holds if we have the converse inequality
\begin{equation}
    \label{mam61}
d_l(x,y) \le d_{l^*}(x,y). 
\end{equation}

Using equalities \eqref{fdes1} and \eqref{fni} we can rewrite \eqref{mam61} in the form
\begin{equation}
    \label{mam62}
\max\{ l(x), l(y)\} \le \inf_{P\in \mathcal{P}_{x,y}}\left(\max_{v\in V(P)} l(v)\right),
\end{equation}
where $\mathcal{P}_{x,y}$ is the set of all paths joining $x$ and $y$ in $G$. Since $x$ and $y$ are vertices of each $P \in \mathcal{P}_{x,y}$, we obtain
\begin{equation}
    \label{mam63}
\max_{v \in V(P)} l^*(v) \ge \max\{l^*(x), l^*(y)\}
= \max\{l(x), l(y)\} 
\end{equation}
for every $P \in \mathcal{P}_{x,y}$.

Now \eqref{mam62} follows from \eqref{mam63}. Thus \eqref{mam61} holds, that implies \eqref{mam59}.
Formulas \eqref{d1} and \eqref{mam59} imply \eqref{uutt-1}. 

The proof is completed. 
\end{proof}

Now we will give an example of $\bf GH$-space generated by finite complete graph.

Recall that a graph $G$ is called {\it complete} if the membership 
\begin{equation*}
    \{u,v\} \in E(G)
\end{equation*}
is valid for every two distinct $u,v \in V(G)$. It is clear that every complete graph is connected.

\begin{Example}\label{hlo}
 Let $G = G(l)$ be a finite complete graph with  non-degenerate labeling $l : V(G) \to \mathbb{R}^+$. Then the following statements are equivalent:
\begin{enumerate}
    \item [(i)] $(V(G), d_l) \in {\bf GH}$ holds, where $d_l$ is an ultrametric on $V(G)$ generated by $G(l)$.
\item[(ii)] The set
\[
l(V) := \{ l(v) : v \in V(G) \}
\]
contains at least $|V(G)| - 1$ points, 
\[
|l(V)| \geq |V(G)| - 1.
\]
\end{enumerate}
\end{Example}

\section{Conclusion --- Expected Results}\label{sec4}

We believe that the following hypothesis is correct.

\begin{Conjecture}\label{far}
Let $i=1,2$ and $T_i=T_i(l_i)$  be finite labeled trees with non-degenerate labelings $l_i : V(T_i) \to \mathbb{R}^+$.
If $(V(T_1), d_{l_1})$ and $(V(T_2), d_{l_2})$ are $\bf GH$-spaces, then the following statements are equivalent:
\begin{enumerate}
    \item[(i)] The equality
\[
D(V(T_1)) = D(V(T_2))
\]
holds, where  $D(V(T_1))$ and $D(V(T_2))$  are  the distance set of the  space $(V(T_1), d_{l_1})$ and $(V(T_2), d_{l_2})$ respectively.  
\item[(ii)] The spaces $(V(T_1), d_{l_1})$ and $(V(T_2), d_{l_2})$ are isometric.
\end{enumerate}
\end{Conjecture}

We note that the implication $(ii) \Rightarrow (i)$ is evidently true  and the validity of $ (i) \Rightarrow (ii)$ follows from the strong triangle inequality if $T_1$ and $T_2$ contain at most three vertices.

The following problem naturally arises.

\begin{Problem}
Describe, up to isometry, all finite $\bf GH$-spaces generated by labeled trees.
\end{Problem}

It was shown in Theorem \ref{tteo5} that every locally finite connected tree with non-degenerate labeling generates an ultrametric space.

\begin{Conjecture}\label{rama} The following statements are equivalent for every connected graph $G$:
\begin{enumerate}
\item[(i)] The ultrametric space $(V(G), d_l)$ is discrete for every non-degenerate labeling
\(
l : V(G) \to \mathbb{R}^+.
\)
\item[(ii)] The graph $G$ is locally finite.
\end{enumerate}
\end{Conjecture}

Recall that an ultrametric space $(X,d)$ is called {\it discrete} if for every $p \in X$ there exists $\varepsilon > 0$ such that
\begin{equation*}
d(p,x) > \varepsilon
\end{equation*}
 for every $ x \in X \setminus \{p\}$.

We note that Conjecture \ref{rama} is valid if $G$ is a tree (see Corollary 5 in \cite{DK2022LTGCCaDUS}).

In Example \ref{hau} we describe a connected graph $H$ that is not locally finite but generates an ultrametric
for each non-degenerate labeling $l : V(H) \to \mathbb{R}^+$.

\begin{Problem} Describe, up to graph isomorphism, all connected graphs $G$ which generate ultrametric spaces for each non-degenerate labeling $l : V(G) \to \mathbb{R}^+$.
\end{Problem}

\begin{Conjecture}\label{pak} Let $G$ be a connected graph. Then the following statements are equivalent:
\begin{enumerate}
    \item[(i)] There exists a non-degenerate labeling $l : V(G) \to \mathbb{R}^+$ such that the pseudoultrametric space $(V(G), d_l)$ is not ultrametric.
\item[(ii)] There exist two distinct vertices $u, v \in V(G)$ and an infinite sequence of cycles $C_n \subseteq G$, $n \in \mathbb{N}$, such that
\(
u, v \in V(C_n) 
\)
for each $ n \in \mathbb{N}$,
\(
\{u,v\} \notin E(G),
\)
and the sets
\(
V(C_{n_1}) \setminus \{u,v\}\) and \(V(C_{n_2}) \setminus \{u,v\}
\)
are disjoint whenever $n_1 \neq n_2$.
\end{enumerate}
\end{Conjecture}

It should be noted here that arguing as in Example \ref{rdwwdd} one can easily prove the validity of the implication $(ii) \Rightarrow (i)$ of Conjecture \ref{pak}.

\section*{Declarations}

\subsection*{Funding} This study was funded by grant 367319 from the Research Council of Finland.

\subsection*{Conflict of interest/Competing interests} The authors declare no conflicts of interest.

\subsection*{Ethics approval and consent to participate} Not applicable

\subsection*{Consent for publication} Not applicable

\subsection*{Data availability} Not applicable

\subsection*{Materials availability} Not applicable

\subsection*{Code availability} Not applicable

\subsection*{Author contribution} Conceptualization, O.D.; writing -- original draft preparation, O.D.; methodology, O.D.; formal analysis, O.R.; writing -- review and editing, O.R.; investigation, O.D.; resources, O.R. All authors have read and agreed to the published version of the manuscript.

\bigskip






\bibliography{bib2020.07}

\end{document}